\documentclass{amsart}
\usepackage{amssymb,epsfig,amsmath,latexsym,amsthm}
\theoremstyle{plain}
\newtheorem{theorem}{Theorem}[section]
\newtheorem{lemma}[theorem]{Lemma}
\newtheorem*{claim}{Claim}
\newtheorem{proposition}[theorem]{Proposition}
\newtheorem{corollary}[theorem]{Corollary}
\newtheorem{question}[theorem]{Question}

\newtheorem{Bounded Diameter Lemma}[theorem]{Bounded Diameter Lemma}
\newtheorem{Non-Encroachment}[theorem]{Non-Encroachment Lemma}

\newtheorem*{canonical*}{Canonical Solid Torus Theorem \ref{2.3}}
\newtheorem*{theorem6.3}{Theorem \ref{6.3}}
\theoremstyle{definition}
\newtheorem{definition}[theorem]{Definition}
\newtheorem{remark}[theorem]{Remark}
\newtheorem{history}[theorem]{History}
\newtheorem{acknowledgements}[theorem]{Acknowledgements}

\newtheorem{remarks}[theorem]{Remarks}

\newtheorem{notation}[theorem]{Notation}

\DeclareMathOperator{\Lim}{Lim}

\DeclareMathOperator{\genus}{genus}

\newcommand{\finv}{f^{-1}}

\newcommand{\BH}{\mathbb H}
\newcommand{\BR}{\mathbb R}

\newcommand{\BN}{\mathbb N}

\newcommand{\mA}{\mathcal{A}}

\newcommand{\mC}{\mathcal{C}}

\newcommand{\mE}{\mathcal{E}}

\newcommand{\mG}{\mathcal{G}}
\newcommand{\mH}{\mathcal{H}}
\newcommand{\mI}{\mathcal{I}}
\newcommand{\mL}{\mathcal{L}}
\newcommand{\mM}{\mathcal{M}}
\newcommand{\mP}{\mathcal{P}}

\newcommand{\mS}{\mathcal{S}}
\newcommand{\mT}{\mathcal{T}}

\newcommand{\PML}{\mP \mM \mL (S)}
\newcommand{\ML}{\mM \mL (S)}
\newcommand{\AML}{\mA \mM \mL (S)}
\newcommand{\LS}{\mL (S)}
\newcommand{\EL}{\mE \mL (S)}

\def\inte{{\text{int}}}
\def\E{\mathcal{E}}
\def\C{\mathcal{C}}

\def\CS{\C (S)}

\begin{document}

\title{Almost filling laminations and the connectivity of ending lamination space}
\author{David Gabai}\footnote{Partially supported by NSF grant
DMS-0504110.}\address{Department of Mathematics\\Princeton
University\\Princeton, NJ 08544}

\maketitle

\setcounter{section}{-1}

\begin{abstract}  We show that if $S$ is a finite type orientable surface of 
negative Euler characteristic which is not the 3-holed sphere, 4-holed sphere or 1-holed torus, 
then $\mE \mL (S)$ is connected, locally path connected and cyclic.\end{abstract}

\section{Introduction}\label{S0}

 The main result of this paper is the following

\begin{theorem}  \label{main} If $S$ is a finite type orientable surface of 
negative Euler characteristic which is not the 3-holed sphere, 4-holed sphere or 1-holed torus, 
then $\mE \mL (S)$, the space of ending laminations, is connected, locally path connected and cyclic.\end{theorem}

Erica Klarrich \cite{K} showed that if $S$ is hyperbolic, then the boundary of the curve complex $\mC (S)$ is homeomorphic to the space of ending laminations on $S$.  Therefore we have the 

\begin{corollary} If $S$ is a finite type orientable surface of 
negative Euler characteristic which is not the 3-holed sphere, 4-holed sphere or 1-holed torus, 
then $\partial \CS $ is connected, locally path connected and cyclic, where $\CS$ is the curve complex of $S$.\end{corollary}

\begin{history}  For the thrice punctured sphere $S$, $\mC(S)$ is trivial.  It is well known that $\partial \mC (S)$ is totally disconnected when $S$ is the 4-holed sphere or 1-holed torus.  Peter Storm conjectured that $\partial \mC (S)$ is path connected if $S$ is any other finite type hyperbolic surface.  See Question 10, \cite {KL}.  Saul Schleimer \cite{Sc1} showed if $S$ is a once 
punctured surface of genus at least two, then  $\mC (S)$ has exactly one end.     With 
Leininger, he  then showed \cite{LS} that $\mE \mL (S)$ is connected if $S$ is any  
punctured surface of genus at least two or $S$ is closed of genus at least 4.   
Leininger, Mj and Schleimer \cite{LMS} showed that if $S$ is a once punctured 
surface of genus at least two then $\mE \mL (S)$ is both connected and locally path 
connected. \end{history}

The idea of the proof is as follows.  Given measured geodesic laminations $\lambda_0$ and $\lambda_1$ whose underlying laminations are minimal and filling, construct a path in $\ML$ from $\lambda_0$ to $\lambda_1$.  A generic PL approximation of this path will yield a new path $f_1:[0,1]\to \ML$ such that for each $t$, $f_1(t)$ is an almost filling almost minimal measured lamination. I.e. it has a sublamination $f_1^*(t)$ without proper leaves whose complement supports at most a single simple closed geodesic.  A measure of the complexity of this lamination is the length of the complementary geodesic, if one exists.  We now find a sequence $f_1, f_2, f_3, \cdots$ such that the minimal length of all complementary geodesics to the $f_n^*(t)$'s $ \to \infty$ as $n\to \infty$.  In the limit, after taking care to rule out compact leaves, we obtain the desired path  in $\EL$ from $\lambda_0$ to $\lambda_1$.  Since the final path can in some sense be taken arbitrarily \emph{close} to the original (see Lemma \ref{effective connectivity}), we obtain local path connectivity.  Proving that the limit path is actually continuous requires control of the passage from $f_i$ to $f_{i+1}$.  Roughly speaking, for each $t$, $f_i(t)$ lies in a $2 \epsilon_i$-neighborhood of $f_j(t)$, if $j>i$, but not necessarily conversely, so as $j\to \infty$, the $f_j(t)$ become more and more complicated.  In actuality we use discrete approximation to control the $f_i$'s. For example, given $t\in [0,1]$ and $i\in \BN$, there exists $t(i)$ such that for $j\ge i$, $f_i(t(i))$ lies in a $2\epsilon_i$-neighborhood of $f_j(t)$.  Neighborhoods are taken in $PT(S)$, the projective unit tangent bundle.

Central to this work, and perhaps of independent interest, is the following elementary result, contained in Proposition \ref{super convergence}, which asserts that the forgetful map $\phi:\ML \to \mL(S)$, from measured lamination space to the space of geodesic laminations is continuous in a super convergence sense.

\begin{proposition} If  the measured laminations $\lambda_1, \lambda_2, \cdots$ 
 converge to  $\lambda \in \ML$, then $\phi(\lambda_1), 
\phi(\lambda_2), \cdots$ super converges to $\phi(\lambda)$ as subsets of $PT(S)$.\end{proposition}

\vskip6pt
As an application (of the proof of  Theorem \ref{main}) in \S8 we give a new construction of non uniquely ergodic measured laminations.

\vskip 8pt

In \cite{RS} Kasra Rafi and Saul Schleimer derived a number of important rigidity results about a surface $S$ and its curve complex $\CS$ under the assumption that $\CS$ is connected.  Combining their results with Theorem \ref{main} we obtain the following two results.
\vskip 8pt

\noindent\textbf{Theorem 9.2} \emph{Let $S$ be a finite type orientable surface of negative Euler chacteristic which is not the 3-holed sphere, 4-holed sphere or 1-holed torus.  Then every quasi-isometry of $\CS$ is bounded distance from a simplicial automorphism of $\CS$.  Consequently, QI$(\CS)$ the group of quasi-isometries of $\CS$ is isomorphic to Aut$(\CS)$ the group of simplicial automorphisms.}  
\vskip 8pt

\noindent\textbf{Theorem 9.3}  \emph{Let $S$ be a finite type orientable surface of negative Euler chacteristic which is not the 3-holed sphere, 4-holed sphere or 1-holed torus.  Suppose that neither $S$ nor $\Sigma$ is the surface of genus-2 or the twice punctured surface of genus-1, then $S$ and $\Sigma$ are homeomorphic if and only if $\CS$ is quasi-isometric to $\mC(\Sigma)$.}
\vskip 8pt

Denote the space of \emph{doubly degenerate Kleinian surface groups} by DD$(S, \partial S)$.    Combining Theorem 6.5 \cite{LS} with our main result we obtain.
\vskip 8pt

\noindent\textbf{Corollary 9.6}  \emph{DD$(S,\partial S)$ is connected, path connected and cyclic if $S$ is a compact hyperbolic surface that is not the sphere with 3 or 4 open discs removed or the torus with an open disc removed.}
\vskip 8pt

\begin{acknowledgements}  We thank Saul Schleimer, Chris Leininger and Kasra 
Rafi for stimulating conversations and beautiful lectures on the curve complex and its 
boundary.  Also Mahan Mj for his constructive comments.  In particular all were independently aware of the structure of the 
balls denoted $B_\alpha$ and their intersections as discussed in \S 2.  We also thank Saul Schleimer and Chris Leininger for making us aware of the very interesting applications of \cite{LS} and \cite{RS}.
\end{acknowledgements}

\newpage
\section{Basic definitions and facts}

In what follows, $S$ will denote a complete, finite volume hyperbolic surface of genus 
$g$ and $p$ punctures, such that $3g+p\ge 5$.  

We will assume that the reader is 
familiar with the elementary aspects of Thurston's theory of curves and laminations on surfaces, e.g. $\mL(S)$ the space of geodesic laminations endowed with the Hausdorff 
topology, $\ML$ the space of measured geodesic laminations endowed with the 
weak* topology, $\PML$ projective measured lamination space, as well as the standard definitions and properties of 
of train tracks, e.g. the notions of recurrent and transversely recurrent tracks,  
branch, fibered neighborhood,  carries, etc.   See \cite {PH}, \cite{H1}, 
\cite{M}, \cite{CEG},  for basic definitions and expositions of these 
results as well as the 1976 foundational paper \cite{T} and its elaboration \cite{FLP} and Thurston's 1981 lecture notes \cite{GT}

All laminations in this paper will be compactly supported in $S$ and isotopic to geodesic ones. For the convenience of the reader, 
we recall a few of the standard definitions.  

\begin{definition}  A lamination is \emph{ minimal} if each leaf is dense.  It 
is \emph{filling} if all complementary regions are discs or once punctured 
discs.  The \emph{closed complement} of a lamination $\mL$ is the metric closure of $S\setminus \mL$ with respect to the induced path metric.  A \emph{closed complementary region} is a component of the closed complement.  A lamination  is \emph{maximal} if each closed complementary region is either a 3-pronged disc or a once punctured monogon.  $\EL$, also known as \emph{ending lamination space} will denote the space of minimal filling geodesic laminations on $S$.     

If $\mL$ is a filling geodesic lamination, then there are only finitely many ways, up to isotopy, 
to extend $\mL$ to another geodesic lamination.  Such extensions, including the trivial one, are called 
\emph{diagonal extensions}. The forgetful map from $\ML$ or $\PML$ to $\LS$ will be denoted by $\phi$.    \end{definition}

Central to this paper are the following definitions.

\begin{definition} An \emph{almost filling} lamination is a lamination $\mL$ 
which is the disjoint union of the lamination $\mL^*$ and the possibly empty simple 
closed curve $L$.  Furthermore, each closed complementary region of $\mL^*$ is either a pronged disc, 
a once punctured pronged disc, a twice punctured pronged disc, or a pronged annulus and there is at 
most  one complementary region of the latter two types.  I.e. if $\mL$ is a geodesic lamination in $S$, then the complement of $\mL^*$ supports at most a single simple closed geodesic.  Note that a filling 
minimal lamination is almost filling.  We say that the almost filling lamination $\mL$ is \emph{almost minimal} if $\mL^*$ has no proper leaves.  We call $\mL^*$  the \emph{almost minimal} sublamination of $\mL$, since $\mL^*$ is the union of one or two minimal laminations.  Note that $\mL^*$ has no compact leaves.  Unless explicitly stated otherwise, all almost filling laminations will be almost filling almost minimal.  
We let $\AML \subset\mL (S)$ denote the set of almost filling almost minimal  geodesic laminations. A path $f:[0,1]\to \ML$ (or $\PML$) will be called \emph{almost filling} if for each $t$, $\phi(f(t))\in \AML$.  \end{definition}

\begin{remarks} It is useful to view a geodesic lamination  in 
three different ways.  First as a foliated closed subset of $S$.  Second, via its lift to  $PT(S)$, the \emph{ projective unit tangent 
bundle}.  Third, let $M_\infty$ denote ($S^1_\infty\times S^1_\infty\setminus \Delta)/\sim$, 
where $\Delta$ is the diagonal and $\sim$ is the equivalence relation generated by 
$(x,y)=(y,x)$.  Since geodesics in hyperbolic 
2-space are parametrized by points in $M_\infty$,  the preimage of a geodesic lamination in $\BH^2$ can be equated with a $\pi_1(S)$-invariant closed subspace of $M_\infty$, but not necessarily conversely.  We view the Hausdorff topology on $\mL(S)$ in these three ways.\end{remarks}

\begin{notation}If 
$x\in \lambda$ and $y\in \mu$ where $\lambda, \mu \in \mL(S)$, then 
$d_{PT(S)}(x,y)$ denotes distance measured in $PT(S)$.    When the context is clear, the subscript will be deleted or even changed to denote distance with respect to another metric.
 If $\mL$ is a geodesic lamination, then $N_{PT(S)}(\mL, \epsilon)$ 
denotes the closed $\epsilon$-neighborhood of $\mL \subset PT(S)$.  The subscript or $\epsilon$ may be deleted or changed as appropriate.

If $X$ is a space, then 
$|X|$ denotes the number of components of $X$.  \end{notation}

\begin{definition}  If $\tau$ is a train track, then metrize $\tau$ by decreeing that each edge has 
length one.  A natural way to split a train track $\tau$ to $\tau_1$, called \emph{unzipping} in the 
terminology of \cite{PH},  is as follows.  Let $N(\tau)$ denote a fibered 
neighborhood of $\tau$ with quotient map $\pi:N(\tau)\to \tau$.    Let $\sigma$ be a union of pairwise 
disjoint, compact, embedded curves in $N(\tau)$ transverse to the \emph{ties} such that each component of $\sigma$ has at least one of its endpoints at 
a singular point of $N(\tau)$.  See the top of Figure 1.7.4 \cite {PH} for an example when $|\sigma|
=2$.  Obtain $\tau_1$, by deleting a small neighborhood of $\sigma$ from $N(\tau)$ and then 
contracting each resulting connected tie to a point, as in Figure 1.7.4 b \cite{PH}.  Call such an unzipping a $\sigma$-unzipping.  Say the unzipping $\tau\to \tau_1$ has length $n$ (resp. length $\ge n$), 
if for each component $\eta$ of $\sigma$, length$(\pi(\eta))\le n$  with equality if exactly 
one endpoint lies on a singularity (resp. length($\pi(\eta)\ge n$, unless $\eta$ has both endpoints on a singular tie.)   
We view the $\sigma$-unzipping to be equivalent to the $\sigma^\prime$-unzipping if there exists a 
tie preserving ambient isotopy, henceforth called a \emph{normal isotopy}, of $N(\tau)$ which takes 
$\sigma$ to $\sigma^\prime$.  We say that the $\sigma^\prime$-unzipping is an extension of the 
$\sigma$-unzipping, if after normal isotopy, $\sigma\subset\sigma^\prime$.  A sequence $\tau_1,\tau_2, \cdots$ is a \emph{full unzipping sequence} if for each $N$, there exists $f(N)$ such that the induced unzipping $\tau_1\to \tau_{f(N)} $ is at least of length $N$.

We say that $\tau$ \emph{fully carries} the measured lamination $ \lambda$, if $\tau$ \emph{carries} $\lambda$ (i.e. up to isotopy $\lambda\subset N(\tau)$ and is transverse to the ties) and each tie nontrivially intersects $\lambda$.\end{definition}

\begin{remarks}  i)  Note that there are 
only finitely many length-$n$ unzippings and in particular only finitely many extensions of a 
length-$n$ unzipping to one of length $n+1$  and every length $\ge n$ unzipping is an extension of a length-$n$ unzipping.  If $\tau_i$ is an unzipping of $\tau$, then $\tau$ 
carries $\tau_i$.  Every splitting or shifting (as defined in \cite{PH}) is also an unzipping. 

ii)  Let $\tau$ be a transversely recurrent train track in $S$.  A theorem of Thurston with proof by Nat Kuhn, as in Theorem 1.4.3 \cite {PH}, is "given $\epsilon>0$, $L>0$, there exists a hyperbolic metric on $S$ such that the geodesic curvature on $\tau$ is bounded above by $\epsilon$ and the hyperbolic length of each edge is at least $L$."   Consequently, with respect to our initial hyperbolic metric, if $\tilde\tau$ is the preimage of $\tau$ in 
$\BH^2$, then each bi-infinite train path $\sigma \subset \tilde\tau$ is a uniform quasi-geodesic in $\BH^2$, hence determines an element of $M_\infty$.  Similarly if $\tau_i$ is carried by $\tau$, then any bi-infinite path in $\tau_i$ is isotopic to one in $\tau$, hence by slightly relaxing the constant, any bi-infinite path carried by a train track carried by $\tau$ is a uniform quasigeodesic.  Recall, that quasi-geodesics in $\BH^2$ are boundedly close to geodescis.  \end{remarks}

\begin{definition} Let $\E(\tau) \subset M_\infty$ denote those geodesics corresponding to bi-infinite train paths carried by $\tilde \tau$.   If $\lambda$ is a geodesic lamination, then let $\mE(\lambda)\subset M_\infty$ denote those geodesics which are leaves of the preimage of $\lambda$ in $\BH^2$. \end{definition}

\begin{remarks}  By construction $\mE (\tau)$ is $\pi_1(S)$-equivariant.   By  Theorem 1.5.4 \cite{PH}  $\E(\tau) $ is closed in $M_\infty$.  \end{remarks}

\begin{proposition}\label{splitting} Let $\lambda_1,\lambda_2 \cdots$,  be a sequence of geodesic measured laminations converging in $\ML$ to the measured lamination $\lambda$.   Let $\tau_1, \tau_2,\cdots$ be a  full unzipping sequence such that $\tau_1$ is transversely recurrent and for each $i$ and for each $ j\ge i, \tau_i$ fully carries $\lambda_j$.  
Then 
\vskip 10pt
\noindent i)  the Hausdorff limit of the 
geodesic laminations $\phi(\lambda_i)$ exists and equals $\mL$, the lamination whose preimage $\tilde\mL\subset\BH^2$ is comprised of the geodesics $\mE =\cap_{i=1}^\infty \mE (\tau_i).$  
\vskip 10pt

\noindent ii) $\phi(\lambda)$ is a sublamination of $\mL$.
   \end{proposition}

\begin{remarks}  This is a generalization of Corollary 1.7.13 
\cite{PH}.  Indeed, if $\lambda=\lambda_1=\lambda_2=\cdots$, then our Proposition is exactly Corollary 1.7.13, though stated in different language.   

In our setting the limit lamination $\mL$ need not carry a measure of full support.  \end{remarks}

\begin{proof}  The proof is close to that of Corollary 1.7.13.    Let $\tilde\tau_i$ (resp. $\tilde\lambda_i$) denote the preimage of $ \tau_i$ (resp. $\lambda_i$) in $\BH^2$.  Since $\tau_1$ is transversely recurrent, any bi-infinite path in $\tilde\tau_1$ or its splittings is a uniform quasi-geodesic.   For $j>i$, we view $\tau_j\subset N(\tau_i)$ embedded and transverse to the ties.

We first show that any geodesic $\gamma\in\mE$ is a limit of a sequence of geodesics $\{\tilde\gamma_i\}$ where $\tilde\gamma_i $ is a leaf of $\phi(\tilde\lambda_i)$ for $i\in \BN$.  By definition $\gamma=\gamma(t_1)=\gamma(t_2)=\cdots$, where $t_i$ is a bi-infinite train path in $\tau_i$ and $\gamma(t_i)$ is the corresponding geodesic.  By Corollary 1.5.3 \cite {PH}, each $t_i$ is normally isotopic in $N(\tau_1)$ to $t_1$.  It therefore suffices to show that each compact segment of $t_1$ is normally isotopic to a segment in $\lambda_i$ for all $i$ sufficiently large.  This follows from the proof of Lemma 1.7.9 \cite{PH}.   That argument shows that if $\tau_i$ is obtained from $\tau_1$ by a length $\ge n$ unzipping;  $\omega$ is a length-$k$, $k\le n/2$, $\tilde\tau_1$-train path normally isotopic in $N(\tilde\tau_1$) to a train path in $\tilde\tau_i$; and $\tau_i$ fully carries the lamination $\lambda_i$, then there exists a segment $\kappa$ lying in a leaf $\tilde\gamma_i$ of $\tilde\lambda_i$ that is normally isotopic in $N(\tilde\tau_1)$ to $\omega$.        (Actually the \cite{PH} argument requires only that $k\le 2n+1$, but using $k\le n/2$ has a somewhat easier proof.)  In our setting if $n\in \BN$, then for $i$ sufficiently large, $\tau_i $ is obtained from $\tau_1$ by a length $\ge n$ unzipping and $\tau_i$ carries $\lambda_i$.  It follows that each element of $\mE$ is the limit of a sequence $\tilde\gamma_1,\tilde\gamma_2,\cdots$ as desired. Since each  $\lambda_i$ is a lamination, this implies that each element of $\mE$ projects to an embedded geodesic in $S$ and the projection of no two elements cross transversely.  Since each $\mE(\tau_i)$ is closed, $\mE$ is closed.   Therefore the $\pi_1(S)$-equivariant $\mE$ projects to a lamination in $S$, which we denote by $\mL$.

Next we show that any convergent sequence of geodesics $\{\tilde\gamma_i\}\subset \BH^2,$ with $\tilde\gamma_i$  a leaf of  $\tilde\lambda_i$, has limit lying in $\mE$.   Suppose $\beta_1, \beta_2,\cdots$ is a sequence of leaves respectively of $  \phi(\tilde\lambda_1)$, $\phi(\tilde\lambda_2),\cdots$ converging to the geodesic $\beta\subset\BH^2$.  We show that $\beta\in \mE$.  Since $\tau_i$ carries $\lambda_j$, $ j\ge i$ it follows that for $j\ge i, \beta_j\subset \mE(\tau_i) $.  Since each $\mE(\tau_i)$ is closed, $\beta\in \mE(\tau_i)$ all $i$ and hence $\beta\in \mE$.

We show that $\mL$ is the Hausdorff limit of the $\phi(\lambda_i)$'s.    Let $U\subset PT(S)$ be a neighborhood of $\mL$.  If for some subsequence, $\phi(\lambda_{n_i})\nsubseteq U$, then there exists a sequence of geodesics $\beta_{n_1}, \beta_{n_2},\cdots$ converging to $\beta\notin \mE$ such that $\beta_{n_i}$ is a leaf of $\phi(\tilde\lambda_{n_i})$.   The previous paragraph shows that  $\cup_{i=j}^\infty \beta_{n_i}\subset \mE(\tau_k)$ if $k\le n_j$, which implies that $\beta\in \mE$, a contradiction. On the  other hand since each leaf of $\mL$ is a limit of leaves of $\phi(\lambda_i)$ and $\mL$ has finitely many minimal components it follows that given $\epsilon>0$, for $i$ sufficiently large $\mL\subset N_{PT(S)}(\phi(\lambda_i),\epsilon$).

To prove ii), note that $\lambda$ is carried by each $\tau_i$ and hence $\mE(\lambda)\subset\mE(\tau_i)$ all $i$ and hence $\mE(\tau_i)\subset \mE$. \end{proof}
The proof of the Proposition yields the following result.

\begin{corollary}  \label{splitting corollary}  Under the hypothesis of Proposition  \ref{splitting}, given $ \epsilon>0$ there exists $N(\epsilon)\in \BN$ so that if $i\ge N(\epsilon)$, then $\hat \mE(\tau_i)\subset N_{PT(S)}(\mL, \epsilon)$, where $\hat\mE(\tau_i)$ denotes the union of geodesics in $S$ corresponding to elements of $\mE(\tau_i)$.\end{corollary}

\begin{proof}

Given $N\in \BN$, then for $i$ sufficiently large, each length $N$ segment lying in a biinfinite train path of $\tilde \tau_i$ is normally isotopic in $N(\tilde \tau_1)$ to a segment lying in a leaf of $\tilde\lambda_i$.  Since bi-infinite train paths in the  $\tilde \tau_j$'s are uniform quasi-geodesics in $\BH^2$ and $\tilde\mL$ is the Haudorff limit of the $\phi(\tilde\lambda_j)$'s , it follows by elementary hyperbolic geometry that for $i$ sufficiently large any geodesic $\sigma$ corresponding to a bi-infinite train path in $\tilde\tau_i$ satisfies $\sigma\subset N_{PT(S)}(\tilde\mL, \epsilon)$.\end{proof}

The curve complex $\CS$ was introduced by Bill Harvey in 1978 \cite{Ha} and it and its 
generalizations have subsequently played a profound role in surface topology, 3-manifold topology and geometry, algebraic topology, hyperbolic geometry and 
geometric group theory.  In particular Howie Masur and Yair Minsky \cite{MM} showed that  $\CS$ is 
Gromov hyperbolic.  See also \cite{B}

Erica Klarrich showed that the boundary $\partial\CS$ of $\CS$ can be equated with 
$\EL$ the space of \emph{ending laminations} of $S$.  These are the geodesic 
laminations on $S$ that are both filling and minimal, with the topology induced from $\ML$ by forgetting the measure.  This work was clarified by 
Ursula Hamenstadt \cite{H1}, who gave a direct combinatorial argument equating ends of quasi-geodesic rays in $\CS$ with points of $\EL$.  In the process, she introduced the \emph{weak Hausdorff topology} 
on $\EL$ (defined below) and showed that $\EL$ with this topology is homeomorphic to $\partial \CS$.  Consequently the two topologies on $\EL$ coincide.

\begin{definition}  (Hamenstadt \cite{H1}) The \emph{weak Hausdorff topology} on $\EL$ is the topology such that a sequence $\mL_1, \mL_2, \cdots$ limits to $\mL$ if and only if 
any convergent subsequence in the Hausdorff topology limits to a diagonal 
extension of $\mL$.  Said another way, the topology is 
generated by the following open sets.  If $\epsilon >0$ and $\mL\in \EL$, then an 
$\epsilon$-neighborhood about $\mL$ in $\EL$ consists of all $\lambda\in \EL$ such that 
$d_H(\lambda,\mL^\prime)<\epsilon$, for some   diagonal extension $\mL^\prime$ of 
$\mL$.

More generally, we say that a sequence $\mL_1, \mL_2, \cdots \in \mL (S)$ converges to the lamination $\mL\in \EL$ with respect to the weak Hausdorff topology if any convergent subsequence in the Hausdorff topology limits to a diagonal 
extension of $\mL.$\end{definition}

The following are three characterizations of continuous paths in $\EL$.

\begin{lemma} \label{path continuity}A function $f:[0,1]\to \EL$ is continuous if 
and only if for each $t\in [0,1]$ and each sequence $\{t_i\}$ converging to $t, f(t)$ 
is the weak Hausdorff limit of the sequence $f(t_1), f(t_2), 
\cdots.$\qed\end{lemma}

\begin{lemma}\label{path connectivity2}  A function $f:[0,1]\to \EL$ is 
continuous if and only if for each $\epsilon>0$ and $s\in [0,1]$ there exists a 
$\delta >0$ such that $|s-t|<\delta$ implies that the maximal angle of intersection 
between leaves of $f(t)$ and leaves of $f(s)$ is $<\epsilon$. \qed\end{lemma}

\begin{lemma}\label{path connectivity3} A function $f:[0,1]\to \EL$ is continuous 
if and only if for each $\epsilon>0$ and $s\in [0,1]$ there exists a $\delta >0$ such 
that $|s-t|<\delta$ implies that  $f(t)\cap N_{PT(S)}(f(s),\epsilon)\neq\emptyset$. 
\qed\end{lemma}

\begin{remark}  Let S be a complete hyperbolic surface that is homeomorphic to a closed surface with punctures.  Let $T$ denote the corresponding compact surface with geodesic boundary, i.e. $T$$\setminus \partial T$ is homeomorphic to $S$.   Then it is well known that there is a natural homeomorphism between $\EL$ and $\mE\mL(T)$.  Similarly the homeomorphism type of $\EL$ does not depend on the hyperbolic structure.  

Therefore, the main result of this paper is purely topological and is applicable to compact surfaces which are not the sphere with 3 or 4 open discs removed or the torus with an open disc removed. \end{remark}

\section{Combinatorics of $\alpha$-balls and almost filling PL 
paths}\label{alpha-balls}

\begin{definition}  If $\alpha$ is a simple closed geodesic 
in $S$, then define $\hat B_\alpha = \{\lambda\in \ML\ |\ i(\alpha,\lambda)=0\}$, where i(.,.) denotes intersection number and let $B_\alpha$ denote the projection of $\hat B_\alpha\setminus 0$ to $\PML$. Let $\hat B_{\alpha,\beta}$ (resp. $B_{\alpha,\beta}$) denote $\hat B_\alpha\cap\hat B_\beta$ (resp. $B_\alpha\cap B_\beta$).  \end{definition}

The main result of this section describes the combinatorial structure of these sets.

\begin{proposition}  \label{PL}If $\alpha$ is a simple closed 
geodesic in $S$, then $B_\alpha$ (resp. $\hat B_\alpha$) is a compact (resp. half open) polyhedral ball of codimension-1 in $\PML$ (resp. $\ML$). If 
$\alpha$ and $\beta$ are distinct simple closed geodesics,  
then $B_{\alpha, \beta}$ (resp. $\hat B_{\alpha,\beta}$) is a compact (resp. half open) polyhedral ball of codimension 
at least two in $\PML$ (resp. $\ML$).\end{proposition}

\begin{remarks} 
i)  Note that $\AML=\phi(\PML\setminus \cup_{\alpha\neq\beta\in \mS} B_{\alpha,\beta})$, where $\mS$ is the set of simple closed geodesics in $S$.

ii)  By Thurston, $\ML$ is an open ball of dimension $6g-6+2m$ and has a natural piecewise integral linear structure (PIL) while $\PML$ is sphere of dimension 
$6g-7+2m$ and has a natural piecewise integral projective structure (PIP). I.e. for $\ML$ the transition functions are piecewise linear functions with integer coefficients.

Here is a brief description of the PIL structure on  $\ML$ as presented 
in \S 2.6 and \S 3.1 \cite{PH}.  A \emph{pair of pants decomposition} on $S$, 
consists  of a collection of $3g-3 +m$  pairwise disjoint simple closed 
geodesics.  Each closed complementary region is either the sphere 
with 3-open discs removed, or the once punctured annulus or the 
twice punctured disc.  Given these curves and a parameterization of 
the complementary regions, construct the associated set of 
\emph{standard train tracks} as in \S 2.6 \cite{PH}.  (From now on 
discussion of the parametrization of the complementary regions will be 
suppressed.) By elementary linear algebra, if $\tau$ is a train track, 
then $V(\tau)$, the space of transverse measures 
supported by $\tau$ is a cone on a compact polyhedron in some $
\BR^N$.  By \S 2.6 \cite {PH} we can identify $V(\tau)$ with a closed 
subspace of $\ML$.  Furthermore, the various $V(\tau)$'s arising from 
the standard tracks glue together to give the PIL-structure on $\ML$.  
In particular the maximal standard train tracks correspond to top 
dimensional cells and have pairwise disjoint interiors.

  This structure is natural in the following sense.  A different pants decomposition  will give rise to a new collection of standard tracks, with change of coordinate maps given by piecewise linear integral maps.  For example, by \cite{HT}, one can transform any pair of pants decomposition to another via a finite sequence of elementary moves. The transition functions associated to  each of the two elementary moves, were computed in Penner's Thesis and are given explicitly in the Addendum of \cite{PH}.  \end{remarks}

\noindent\emph{Proof of Proposition} \ref{PL} Since $\ML$ has a 
natural PIL structure, it suffices to show that $B_\alpha$ and $\hat B_
\alpha$ are polyhedral balls with respect to some pants 
decomposition.  Let $\mP$ be a pants decomposition with curves $
(\alpha_1,\alpha_2,\cdots,\alpha_{3g-3+m})$, where $\alpha_1=\alpha
$.  Let $\lambda$ be a geodesic measured lamination.  If  
$\alpha$ is either a leaf of $\lambda$ or lies in a complementary region,  
then with respect to \emph{global coordinates} $\lambda$ corresponds 
to a point $(0,t_1,i_2,t_2, i_3, t_3, \cdots, i_{3g-3+m}, t_{3g-3+m})$, 
where $i_j\ge 0$, and $t_j\in \BR$, unless $i_j=0$ in which case $t_j
\ge 0$.  In particular $t_1\ge 0$.   Here $i_j(\lambda)=i(\alpha_j,\lambda)$ and $t_j(\lambda)$ is the twisting of $\lambda$ about $\alpha_j$.  For a detailed exposition of global coordinates see p. 152, Theorem 3.1.1  and the last paragraph of p. 
174 of \cite{PH}.  The connection between these coordinates and the 
homeomorphism between $\ML$ and $\BR^{6g-6+2m}$ is given in 
Proposition 2.6.3 and Corollary 2.8.6 \cite{PH}.  

The collection of 
points in $\ML$ with first and second coordinate zero is the set of  
measured geodesic laminations disjoint from $\alpha$.  It is an open PL 
ball, denoted $\hat S_\alpha$, of dimension $6g-6+2m-2$ and is the 
union of those $V(\tau)$'s for which $\tau$ is a standard 
train track disjoint from $\alpha$.  Each element of $\hat B_\alpha$ is 
of the form  $t\alpha +\mG$, where $t\ge 0$ and  $\mG$ is a 
measured geodesic lamination with support disjoint from $\alpha$,  
thus is a half open PL-ball of codimension-1, i.e. it is homeomorphic to $\BR^{6g-6+2m-2}\times[0,\infty)=\hat S_\alpha\times [0,\infty)$.    

The quotient of $\hat S_\alpha\setminus 0$ by projectivizing is a sphere of dimension $6g-7+2m-2$ denoted  $S_\alpha$.  Therefore, $B_\alpha$ is a cone on $S_\alpha$ and is thus a compact PL ball of codimension-1 in $\PML$. 

If $\beta\cap\alpha =\emptyset$, then we explicitly describe $\hat B_
{\alpha,\beta}$  as follows.  Let $\mP$ be a pants decomposition with 
$\alpha_1=\alpha$ and $\alpha_2=\beta$.  If $\lambda\in \hat B_
{\alpha,\beta}$, then $ \lambda$ has coordinates $(0,t_1,0,t_2, i_3,t_3, 
\cdots, i_{3g-3+m}, t_{3g-3+m})$.  As in the previous paragraph, when  
$t_1=t_2=0$, this gives rise to an open PL-ball $\hat S_{\alpha,\beta}$ of 
codimension-4 in $\ML$  and since $t_1, t_2\ge 0$, we conclude that $
\hat B_{\alpha,\beta}$ is a half open PL ball of codimension-2.  Similarly the quotient 
$S_{\alpha,\beta}$ of $\hat S_{\alpha,\beta}\setminus 0$ is a sphere of codimension-4 and  $B_{\alpha,\beta}
$  is the join of a sphere and an interval (the projective classes supported on $\alpha_1\cup\alpha_2$), hence is a compact PL-ball of 
codimension-2.

If $\beta\cap\alpha\neq\emptyset$, then let $\mP$ be a pants decomposition where for some $n\ge 2,\newline \{\alpha_n,\alpha_{n+1},\cdots, \alpha_{3g-3+m}\}$ is a maximal system of simple closed geodesics disjoint from both $\alpha$ and $\beta$.  We can assume that $\alpha_n,\alpha_{n+1},\cdots, \alpha_{k}$ are the curves that can be isotoped into $ N_S(\alpha\cup\beta)$.  Thus $B_{\alpha,\beta}$ is the join of a PL sphere of dimension $6g-7+2m-2k$ with a ball of dimension $k-n$, hence is a compact ball of codimension $n+k-1\ge 3$. A similar result holds for $\hat B_{\alpha,\beta}$.\qed\vskip 10 pt
The proof of the Proposition shows. 

\begin{corollary}  If $\alpha$ is a simple geodesic, then $B_\alpha$ is a cone of PL-sphere.  If $\lambda\in \hat B_\alpha$, then $\lambda=t\alpha + \mG$, where $\mG$ is a measured lamination with support disjoint from $\alpha$ and $t\ge 0$.\end{corollary}

\begin{remark}  A similar description exists for elements of $B_{\alpha,\beta}$ and $\hat B_{\alpha,\beta}$.\end{remark}

The above facts about the topology and combinatorics of the $B_\alpha$'s and their intersections is independently more or less known to the experts.

\begin{definition} \label{escape ray} If $x \in \inte(\hat B_\alpha)$ with $\phi(x)\in \AML$, then $ \phi(x)$ is the disjoint union of an almost minimal almost filling 
lamination $\mL^*$ and the simple closed curve $\alpha$.   Let $m_{\mL^*}$ (resp. $m_\alpha$) denote the 
transverse measure on $\mL^*$ (resp. $\alpha$).   Define the \emph{escape ray} $r_x$ to be  the path  $r:[0,1]\to \hat B_\alpha$ by $r(t)=(\mL^*, m_{\mL^*}) 
+ (\alpha, (1-t)m_\alpha)$.  So $r(0)=x$,  $\phi(r(1))=\mL^*$ and for each $t<1,\ \phi(r(t))= \mL^*\cup \alpha$.  \end{definition}

\begin {definition}  A path $g:[0,1]\to \ML$ (resp. $\PML$) is \emph{PL almost filling} if $g$ is piecewise linear, transverse to each $\hat B_\alpha$ (resp. $B_\alpha$) and disjoint from each $\hat B_{\alpha,\beta}$ (resp. $B_{\alpha,\beta}$) and hence $\phi(g(t))\in \AML$ all $t$.  Also $\phi(g(t))\in \EL$ for $t=0,1$ as well as $\phi$ of  the vertices of $g$.    We say that $g$ is \emph{ filling} (resp. \emph {almost filling}) if $\phi(g(t) )\in \EL$ (resp. $\AML$) for all $t$.  \end{definition}

\begin{lemma}  \label{PL path}If $x_0, x_1 \in \ML$ with $\phi(x_0), \phi(x_1)\in \EL$, then there exists a PL almost filling path from $x_0$ to $x_1$.\end{lemma}

\begin{proof}  Since  the $B_\alpha$'s  countable,  PL of codimension-1 and the $B_{\alpha,\beta}$'s are countable, PL of codimension at 
least two, it follows that the generic PL path is almost filling.\end{proof}

Since generic PL paths are PL almost filling we have

\begin{lemma} \label{generic PL path} Fix any  metric on $\ML$ and $ \epsilon>0$.  Let $ h:[0,1]\to \ML$ be continuous, such that 
$\phi(h(t))\in\mE\mL(S)$ for $i=0,1$.  Then $h$ is homotopic  
rel endpoints to a PL almost filling path $g$ such that for each $t, d_{\ML}(g(t), 
h(t))<\epsilon$.\qed\end{lemma}

\section{super convergence}

The forgetful map $\phi:\ML\to \LS$ is 
discontinuous, for any simple closed curve viewed in $\ML$ is the limit of 
filling laminations and any Hausdorff limit of a sequence of filling laminations 
is filling.  The content of this section is that this map is continuous in a 
\emph{super convergence} sense.

\begin{definition}  Let $X_1, X_2, \cdots$ be a sequence of subsets of the 
topological space $Y$.  We say that the subsets $\{X_i\}$ \emph{super converges} to $X$ 
if for each $x\in X$, there exists $x_i\in X_i$ so that lim$_{i\to \infty} x_i\to x$.  
In this case we say $X $ is a \emph{sublimit} of $\{X_i\}$.\end{definition}

The following is the main result of this section.

\begin{proposition} \label{super convergence} i)  If  the measured laminations $\lambda_1, \lambda_2, \cdots$ 
 converge to  $\lambda \in \ML$, then $\phi(\lambda_1), 
\phi(\lambda_2), \cdots$ super converges to $\phi(\lambda)$ as subsets of $PT(S)$.  \vskip8pt

\noindent ii)  If in addition $\phi(\lambda)\in \AML$ and $\phi(\lambda_i)\in \AML$ all $i$, then $\mL^*_1,\mL^*_2, \cdots$ super 
converges to $\mL^*$, where $\mL^*_i$ (resp. $\mL^*$) denotes the almost minimal 
almost filling sublamination of $\phi(\lambda_i$) (resp. $ \phi(\lambda))$. 
\end{proposition}

\begin{proof}  Part i) follows immediately from the following Claim and Proposition \ref{splitting}.

\begin{claim} After passing to subsequence of $\{\lambda_i\}$ there exists a full 
unzipping sequence $\tau_1, \tau_2, \cdots$ of transversely recurrent 
train tracks  such that each $\tau_i$ carries $\lambda$ and for $j\ge i$, $\tau_i$ fully carries $\lambda_j$. \end{claim}

\noindent\emph{Proof of Claim.}There are finitely many transversely recurrent train tracks such that each $x\in \ML$ is carried by one of them, e.g. a system of standard train tracks.  Therefore one such track $\tau_1$ carries 
infinitely many of the $\lambda_j$'s.   By deleting edges if necessary we can assume that 
$\tau_1$ fully carries infinitely many of these measured laminations and, after passing to 
subsequence, these laminations comprise our original sequence.   Since each $V(\tau_1)$ is a closed subspace of $\ML$ (see Remark 2.3) and $\lambda$ is the limit of the $\lambda_i$'s it follows  that $\tau_1$ 
also carries $\lambda$.

Consider the finitely many train tracks obtained by a length-$1$ unzipping of $\tau_1$.  Again one such track $\tau_2$ fully carries an 
infinite subset of $\{\lambda_i\}$, $i\ge 2$ and hence also $\lambda$.  The claim 
follows by the usual diagonal argument. \qed 
\vskip 10pt

We now prove part ii).    First consider the case that $\mL^*$ is minimal.  Let $I$ be a compact interval lying in a leaf of $\mL^*$.  We 
first show that $I$ is a sublimit of $\mL^*_i$'s if and only if $\mL^*$ is such a 
sublimit.  By inclusion,  the latter implies the former.  Conversely if 
$\mL^*$ is not a sublimit, then there exists an $x\in \mL^*\subset PT(S)$ which is not 
a limit point of some subsequence of the $\mL^*_i$'s.  But the segment $I$ 
lies in the leaf $ \sigma$ which is dense in $ \mL^*$.  Thus, if some leaf of $\mL^*_i$ is nearly tangent to 
$I$, then in $S$ it must pass very close and tangent to $\mL^*$ at $x$, a contradiction.

If $I$ is not a sublimit of $\mL^*_i$'s, then after passing to subsequence, 
there exists an open set $U\subset S$ such that $I\subset U$ and for each $i$, $U\cap \mL^*_i=\emptyset$.  Otherwise, the $\mL^*_i$'s would intersect $I$  with some definite angle, but this contradicts the fact that  $I$ is a sublimit of the $\phi(\lambda_i)$ laminations.  This argument applied to all compact intervals in leaves of $\mL^*$ plus compactness of $\mL^*$
 implies that there exists an 
open set $V\subset S$ such that $\mL^*\subset V$ and $V\cap\mL^*_i=\emptyset$ all $i$ 
sufficiently large.  
This contradicts the fact that the complement of an almost filling lamination cannot support an almost filling lamination.  Since $\mL^*$ is minimal without compact leaves, there exists infinitely many simple closed geodesics in $S$ that lie in $V$.  This contradicts the fact that the complement of each $\mL_i^*$ supports at most a single geodesic.

The other possibility is that $\mL^*$ is the union of two minimal components $\mH_1$ and $\mH_2$.  As above it suffices to show that if  for $i=1,2$  $I_i$ is a compact interval in a leaf of $\mH_i$, then $I_1\cup I_2$ is a sublimit of the $\mL^*_i$'s.  If say $I_1$ was not a sublimit, then that argument also shows that there exists an open set $V\subset S$ such that $\mH_1\subset V$ and for $i$ sufficiently large $V\cap \mL_i^*=\emptyset$.   Since $V$ supports infinitely many simple closed geodesics we obtain a contradiction as above.  \end{proof}

An immediate application is the
following discrete approximation Lemma.

\begin{lemma} \label{discrete approximation} Let $f:[0,1]\to \ML$ be an almost filling path.  Let $\epsilon>0$ 
and $F\subset [0,1]$ a finite set.  Then there exists an open cover $\mI$ of $[0,1]$ by connected sets $I(1),\cdots, I(n)$ and $t_1=0<t_2<\cdots<t_{n}=1 $
such that $F\subset \{t_1, t_2, \cdots, t_{n}\}$ and  $t_i\in I(j)$ exactly when  $i=j$.  Finally for all $t\in I(i)$, $ \mL^*(t_i)\subset 
N_{PT(S)}(\mL^*(t),\epsilon)$.  Here $\mL^*(t)$ (resp. $\mL^*(t_i)$) denotes the almost minimal almost filling 
sublamination of $f(t)$ (resp. $f(t_i))$. \qed \end{lemma}

\section{A criterion for constructing continuous paths in $\EL$}

Recall that our compact surface $S $ is endowed with a fixed hyperbolic metric 
$\rho$.  Let $\{C_i\}_{i\in \BN}$ denote the set of simple closed geodesics in $S$, 
ordered so that $i<j$ implies length$_\rho(C_i)\le$ length$_\rho(C_j)$.   

\begin{definition}  A \emph {pointed open covering} $\mT=(T, \mI)$ of $[0,1]$ is a 
set $T= \{t_1=0<t_2<\cdots< t_n=1\}\subset [0,1]$ and an open covering $\mI$ of $[0,1]$ by  connected 
sets  $I(1), I(2), \cdots, I(n)$ such that $\ t_i\in I(i)$ and if $I_i(j)\cap I_i(k)\neq\emptyset$ then $|j-k|\le 1$.    A \emph{refinement} $(T_2,\mI_2)$ of $(T_1,\mI_1)$ 
has the property that $T_1\subset T_2$ and each $I_2(j)$ is contained in some $I_1
(k)$. The elements of $T_i$ (resp. $\mI_i$) are denoted $t^i_1,\cdots, t^i_{n_i}$, (resp. $I_i(1), \cdots, I_i(n_i)$).  An \emph{infinite pointed open covering sequence} $
\mT_1,\mT_2,\cdots$ is a sequence of pointed open covers such that each 
term is a refinement of the preceding one.  An \emph{assignment function} $\mA$ of an infinite pointed open covering sequence associates to each $s\in [0,1]$ a nested sequence of open sets $I_1(j_1)\supset I_2(j_2)\supset\cdots$  with $s\in \cap_{i=1}^\infty I_i(j_i)$ and $I_i(j_i)\in \mI_i$.    We  let $s(i)$ denote $t^i_{j_i}$.  \end{definition} 

\begin{remark} It is an elementary exercise to show that any infinite pointed open covering sequence has an assignment function.\end{remark}

The following very technical lemma gives a criterion for constructing a 
continuous path in $\EL$ between two elements of $\EL$.  

\begin{lemma} \label{path criterion} Let $f_i:[0,1]\to \ML$, $i\in \BN$ be a 
sequence of almost filling paths between   two given points in $\EL$.  Let $\epsilon_1, \epsilon_2, \cdots$ be such that for all $i$, 
$\epsilon_i/2>\epsilon_{i+1}>0$  and let   $\mT_1,\mT_2,\cdots$ be an infinite pointed open cover sequence with assignment function $\mA$.   Let $\mL_n^*(s)$ 
denote the almost minimal almost filling sublamination  of $\phi( f_n(s))$.    
\vskip 10pt

\noindent \textrm{(sublimit)}  If $I_i(j)\cap I_{i+1}(k)\neq \emptyset$ and $t\in I_{i+1}(k)$, \newline then  $ \mL_i^*(t_j^i)\subset 
N_{PT(S)}(\mL_{i+1}^*(t),\epsilon_i)$

\vskip 10pt

\noindent (\text{filling})  If $s\in [0,1]$ and $m\in \BN$, then for $n$ sufficiently large the 
minimal angle of intersection between leaves of $\mL_n^*(s(n))$ with $C_m$ is 
uniformly bounded below by a constant that depends  only on $s$ and $m$.    Also for $m$ fixed, any
subsequence of $C_m\cap \mL_i^*(s(i)), i\in \BN$, has a further 
subsequence which converges in the Hausdoff topology on $C_m$ to a set with at 
least $4g+p+1$ points. Here $g=genus(S)$ and $p$ is the number of punctures. (Recall that $s(n)$ is given by $\mA$.)

\vskip 10pt

If all the above hypotheses hold, then there exists a continuous path $\mL:[0,1]\to \EL$ connecting $f_1(0)$ to $f_1(1) $
so that   for $s\in [0,1]$, $\mL(s)$ is the weak Hausdorff limit of $\{\mL_n^*(s(n))\}_{n\in \BN}$.
\end{lemma}

\begin{proof}  We record the following useful and immediate fact which is a consequence of the sublimit condition and the requirement that for all $i$, $\epsilon_{i+1}<\epsilon_i/2$.

\begin{claim}  If $t\in I_i(j)$  and $n\ge i$, then $\mL_i^*(t_j^i)\subset N_{PT(S)}(\mL_n^*(t) ,2\epsilon_i)$. \qed\end{claim}

This claim demonstrating the utility of the sublimit condition shows how a given $f_i$ imposes structure on $f_k$ for all $k\ge i$ and is the key to the proof of continuity.   A similar result together with the filling condition are the main ingredients for showing that $\mL(s)$ is the weak Hausdorff limit of the $L^*_n(s(n))$'s.  Here are the details.

Fix $s$.  We construct the minimal and filling $\mL(s)$.  Given $s\in [0,1]$, let 
$I_1(j_1)\supset I_2(j_2)\supset \cdots$ be the nested sequence of $I_i(j)$ intervals given by $\mA$
which contain $s$.   After passing to subsequence we can assume that 
$\{\mL_{n_i}^*(s(n_i))\}$ converges in the Hausdorff topology to the lamination 
$\mL^\prime$.  $\mL^\prime$ contains no closed leaves since the filling condition implies that $\mL^\prime$  is transverse to every simple closed geodesic.  Let $\mL$ be a minimal sublamination of $\mL^\prime$.  If 
$\mL$ is not filling, then there exists a simple closed geodesic $C$, disjoint from 
$\mL$ that can be isotoped into any neighborhood of  $\mL$ in $S$.    An 
elementary topological argument shows that  $|C\cap \mL^\prime| < 4g +p +1$.  This 
contradicts the second part of the filling condition. We now show that $\mL$ is 
independent of subsequence.    Let $\mL_1\in \EL$ be a  lamination that is the weak Hausdorff limit of the subsequence $\{\mL^*_{m_i}(s(m_i))\}$ and let
$\epsilon>0$.  If $x\in \mL_1$, then there exists a sequence $x_{m_i}\in 
\mL_{m_i}^*(s(m_i))$ such that lim$_{i\to \infty}\ x_{m_i}=x$ where the limit is taken in 
$PT(S)$.  The sublimit property implies that for all $k>m_i$, there exists $y_k\in 
\mL_k^*(s(k))$ such that $d_{PT(S)}(x_{m_i}, y_k)<\epsilon_{m_i}+\epsilon_{m_i+1}+\cdots 
+\epsilon_{k-1}<2\epsilon_{m_i}$.   As $i\to \infty$, $\epsilon_{m_i}\to 0$ and $n_i\to \infty$ so it follows that  $\mL_1\subset \mL^\prime$.  Since $\mL_1, \mL\in \EL$, we conclude that $\mL=\mL_1$.   Denote $\mL$ 
by $\mL(s)$.   We have shown that $\mL(s)$ exists and satisfies the weak Hausdorff limit condition of the conclusion.

We apply Lemma \ref{path connectivity3} to show that $ f$ is continuous.  Fix $s\in 
[0,1]$ and let  $I_1(j_1)\supset I_2(j_2)\supset \cdots$ be the nested sequence of 
$I_i(j)$ intervals given by $\mA$ which contain $s$.  
    Let $U $ open in $PT(S)$ such that $\mL(s)\subset 
U$.  Pick $z\in \mL(s)$ and let $\epsilon >0$ such that $N_{PT(S)}(z,\epsilon)\subset U$.  We will show that for $t$ sufficiently close to $s$, $\mL(t)\cap N_{PT(S)}(z,\epsilon)\neq\emptyset$.
Choose $i$, such that $\epsilon_i<\epsilon/2$.  Since $\mL(s)$ is the weak Hausdorff 
limit of the $\mL_k^*(s(k))$ it follows that for $k$ sufficiently large there exists an 
$x_k\in \mL_k^*(s(k))$ so that $d_{PT(S)}(x_k,z)<\epsilon_i$. Assume 
that $k>i$.   The Claim implies that if $t\in I_k(j_k)$, then for every $n
\ge k$ there exists $y_n\in \mL^*_n(t)$ such that $d_{PT(S)}(x_k,y_n)\le 2
\epsilon_k$ and hence for some diagonal extension $\mL^\prime(t)$ of 
$\mL(t)$ there exists $y \in \mL^\prime(t)$ with $d_{PT(S)}(y, x_k)\le 2
\epsilon_k<\epsilon_i$ and hence  
$d_{PT(S)}(y,z)<2\epsilon_i<\epsilon$.

Assume that $\epsilon$ is sufficiently small so that the $\epsilon$-disc $D\subset S$, centered at $z$, is embedded.  
Since $\mL(s)$ is minimal and has infinitely many leaves, the above argument shows that given $N\in \BN$, there exists a $\delta>0$ so that $|t-s|<\delta$ implies for some diagonal extension $\mL^\prime(t)$ of $\mL(t)$, $\mL^\prime(t)\cap D$ contains at least $N$ arcs extremely close to and nearly parallel to the local leaf of $\mL(s)$ through $z$.  
Extremely close and nearly parallel mean that if $\sigma$ is a geodesic path lying between a pair of these arcs, then $\sigma\cap N_{PT(S)}(z,\epsilon) \neq\emptyset$.  It is possible that all of these $N$ arcs lie in $\mL^\prime(t)\setminus\mL(t)$, however, if $N$ is sufficiently large and $\gamma\subset D$ is a geodesic  transverse to these $N$ arcs, then $\gamma\cap \mL (t)  \neq\emptyset$.  Therefore, $\mL(t)\cap N_{PT(S)}(z,\epsilon)\neq\emptyset.$   \end{proof}

\section{$\EL$ is path connected}

Let $\mu_0$, $\mu_1\in \EL$.   In this section we produce a sequence of PL almost filling paths in 
$\ML$ that satisfy  Proposition \ref{path criterion}, from 
which we obtain a path from $\mu_0$ to $\mu_1$.  

By Lemma \ref{PL path} there exists a PL almost filling  path $f_1:[0,1]\to \ML$ from $\mu_0$ to $\mu_1$.   Fix $\epsilon_1>0$.  Let $\mL_1^*(t)$ denote the almost minimal almost filling 
sublamination of $\phi(f_1(t))$.  By Lemma \ref{discrete approximation}, there 
exists a pointed open cover $\hat \mT_1=(T_1,\hat \mI_1)$ with  $T_1=\{t_1^1,\cdots, t^1_n\}$ and $\hat \mI_1=\{\hat I_1(1),\cdots , \hat I_1(n)\}$ such that for $t\in \hat I_1(j)$, 
$\mL_1^*(t_j^1)\subset N(\mL_1^*(t), \epsilon_1/2)$.  By shrinking the $\hat I_1(j)$'s we can assume that for each $k$ some open neighborhood $Y(k)$ of $t_k^1$ nontrivially intersects only $\hat I_1(k)$.  By further shrinking the $\hat I_1(k)$'s we get a new open cover   $\mI_1=\{I_1(1),\cdots,I_1(n)\}$ and $\eta_1>0$ such that for all $j$, $N_{[0,1]}(I_1(j),\eta_1)\subset \hat I_1(j)$.  We let $\mT_1=(T_1,\mI_1)$.    Since $f_1$ is transverse 
to the $\hat B_\alpha$'s, it intersects $\hat B_{C_1}$ in at most a finite set of points.  By including these points in $F$ (notation as in Lemma \ref{discrete approximation}) we 
can assume that these points were included in $T_1$.     Without loss we will assume that there exists a unique point $x$ of intersection, 
and $f_1(t_p^1)=x$.

We  homotope $f_1$ to $f_2$ so that $f_2\cap \hat B_{C_1}=\emptyset$ as follows.    First 
homotope $f_1$ to $f_{1.1}$, via a homotopy supported in $Y(p)$ so 
that the image of $f_{1.1}$ is the image of $f_1$ together with the 
escape ray $r_x$. (See  Definition \ref{escape ray}).  I.e. the path $f_
{1.1}$ follows along the image of $f_1$ until it hits $x$, then  goes all the way out along $r_x$ and then back along $r_x$ to 
 $x$ and then continues away from $\hat B_{C_1}$ as does $f_1$.  Since $\hat B_{C_1}$ 
is  PL ball of codimension-1 a very small perturbation of $f_{1.1}$ yields 
$f_{1.2}$  disjoint from $\hat B_{C_1}$.  A generic PL approximation of $f_{1.2}$   yields a PL almost filling path $f_2$ disjoint from $\hat B_{C_1}$.  

Since $\mL_1^*(t_p^1)$ is the almost minimal almost filling lamination associated to any point 
on $r_x$, Proposition \ref{super convergence} plus compactness implies that if the nontrivial tracks of the 
homotopy from $f_{1}$ to $f_{2}$ lay sufficiently close to $r_x$, then  for $s\in Y(p)$, 
 $\mL_1^*(t_p^1)\subset N(\mL_{2}^*(s),\epsilon_1)$.  Since 
$f_2([0,1])\cap \hat B_{C_1}=\emptyset$, $C_1$ nontrivially intersects each $\mL_2^*(t)$ 
transversely.  
Since the homotopy from $f_1$ to $f_2$ is supported in $Y(p)$, $f_2$ satisfies the 
sublimit property  (i.e. the sublimit property holds for $i+1=2$) provided that the, to be defined, cover $\mI_2$ has the property that if $I_2(j)\cap I_1(k)\neq \emptyset$, then $I_2(j)\subset \hat I_1(k)$.

We now choose $\epsilon_2$ so the (to be constructed) $f_i$'s will satisfy the 
filling property.   By compactness of $[0,1]$ and  Proposition 
\ref{super convergence} there exists a $\psi>0$ that is a uniform lower bound, independent of $t$, 
for the  maximal angle of intersection between a 
$\mL^*_2(t)$ and $C_1$.  Furthermore, 
there exists $\kappa>0$ so that if $\mL$ is a geodesic lamination, $t\in [0,1]$ and $\mL_2^*(t)\subset N( 
\mL,\kappa)$, then $\mL\cap C_2\neq\emptyset$ and the maximal angle of intersection is at least $\psi/2$.    
Since $C_1$ has bounded length, the  lower bound $\psi/2$
 on maximal angle of intersection of a geodesic lamination with $C_1$  implies a lower bound $\phi$ on 
minimal angle of intersection.

Let $N=4g+p+1$, where $g=\genus(S)$ and $p$ is the 
number of punctures.  Since for every $t$, $|\mL^*_2(t)\cap C_2|=\infty$,  the angle condition, compactness, and Proposition \ref{super convergence} imply that
there exists a $\psi^\prime>0$ and $\kappa^\prime<\kappa$ so that each $\mL_2^*(t) \cap \C_1$ contains $N$ points, 
any two of which are at least  distance $\psi^\prime$ apart, measured along $C_1$.  
Furthermore, if  $\mL$ is a geodesic lamination  with $\mL_2^*(t)\subset N( 
\mL,\kappa^\prime)$, then $\mL\cap C_1$ contains $N$ points, any two of which are 
$\psi^\prime/2$ apart on $C_1$.     Pick $\epsilon_2<\min(\kappa^\prime/2, 
\epsilon_1/2)$.   

If the, to be defined, $f_i$'s and  $I_i(j)$'s satisfy the sublimit property and the  
$\epsilon_i$'s are chosen so that $\epsilon_{i+1}<\epsilon_i/2$, then the $f_i$'s will 
satisfy the filling property with respect to $C_1$, independent of choice of assignment function.  To see this, first pick $s\in [0,1]$ and suppose that $s\in I_2(q)$.  If $m\ge 2$ and $s\in I_m(j)\in \mI_m$, then by the sublimit property, $\mL^*_2(t_q^2)\subset N_{PT(S)}(\mL^*_m(t^m_j),\epsilon_{m-1}+\cdots+\epsilon_{3}+\epsilon_2)\subset N_{PT(S)}(\mL^*_m(t^m_j),2\epsilon_2)\subset N_{PT(S)}(\mL^*_m(t^m_j),\kappa^\prime)$ and hence $\mL^*_m(t^m_j)\cap C_1$ contains $N$ points spaced in $C_1$ at least $\psi^\prime/2$ apart and the minimal angle of intersection of $\mL^*_m(t^m_j)$ with $C_1$ is bounded below by $\phi$.  The filling property for $C_1$ follows.   

Let $X_2=\{t\in [0,1]|f_2(t)\cap \hat B_{C_2}\neq\emptyset\}$.
Apply Lemma \ref{discrete approximation} to $f_2$ using $\epsilon= 
\epsilon_2$ and $F=T_1\cup X_2$ (henceforth called $F_2$) to obtain $\hat \mT_2=(T_2,\hat\mI_2)$.  Again we can assume that for each $k$, some open neighborhood of $t^2_k$ nontrivially intersects only $\hat I_2(k)$.  By adding appropriate extra points to $F_2$, $\hat I_2(j)\cap I_1(k)\neq\emptyset$ implies $\hat I_2(j)\subset \hat I_1(k)$.  Finally shrink $\hat \mI_2$ as before to obtain $\mI_2$ and $\mT_2$.

 Use $\epsilon_2$ 
and $C_2$ to construct $f_3$ from $f_2$ in the same manner that $f_2$ was constructed 
from $f_1$. 
Again the resulting $f_3$ satisfies the sublimit condition provided $I_3(j)\cap I_2(k)\neq\emptyset$ implies $I_3(j)\subset \hat I_2(k)$.  Choose 
$\epsilon_3$ as above so that subsequently constructed $f_i$'s will satisfy the filling condition with respect to $C_2$.  Define and construct $\hat \mT_3$ and $\mT_3$ analogous to the constructions of $\hat\mT_2$ and $\mT_2$. 

The proof of path connectivity is now completed by induction.\qed 

\vskip 10pt

Our proof shows that in the sublimit sense any path in $\ML$ connecting points in $\EL$ can be perturbed rel endpoints by an arbitrarily small amount to a path in $\EL$.  More precisely we have, 

\begin{lemma}\label{effective connectivity}  If $f:[0,1]\to \ML$ is a PL almost filling
path  with $\phi(f(0), \phi(f(1))$ in $\EL$, $\epsilon>0$ and $\delta>0$, then there exists a 
path $g:[0,1]\to \EL$ from $\phi(f(0))$ to $ \phi(f(1))$ such that for each $s\in [0,1]$ there exists $t\in [0,1]$ with $|t-
s|<\delta$ such that, 
 $\mL^*(t)\subset N_{PT(S)}(g(s)^\prime,\epsilon)$, for some diagonal extension $g(s)^\prime$ of $g(s)$.  Here $\mL^*(t)$ denotes the almost minimal almost filling sublamination of $\phi(f(t))$.\end{lemma}
 
 \begin{proof}  Apply the above proof with $f_1=f$, $\epsilon_1<\epsilon/4$ and the mesh of $F<\delta$.   Pick $s\in [0,1]$.  With  notation as in  Proposition \ref{path criterion},    if $n$ is sufficiently large  $\mL_n^*(s(n))\subset N_{PT(S)}(\mL^\prime(s),\epsilon_1)$   for some diagonal extension $\mL^\prime(s)$ of $ \mL(s)$.  By the sublimit property $\mL_1^*(s(1))\subset N_{PT(S)}(\mL_n^*(s(n)),2\epsilon_1)$.  Since $\delta\le \textrm{mesh}\ T_1$,   $|s-s(1)|\le \delta$ and the result follows.\end{proof}

\section{$\EL$ is locally path connected} 

In this section we show that if S satisfies the hypothesis of Theorem 
\ref{main}, then $\EL$ is locally path connected.  It suffices to show that if 
$\mu\in U$  an open set in $\EL$, then there exists an open set $V\subset U$ 
containing $\mu$ such that for each $\mL\in V$ there is a path in $U$ from $\mu$ to $\mL$.    Therefore, the path component $W$ of $U$ containing $\mu$ is open and path 
connected.

In order to free the reader of nasty notation and slightly extra technical detail, we first 
consider the case that $\mu$ is maximal, i.e. every complementary region is 
either a 3-pronged disc or a once punctured monogon.  

Let $\tau_1$ be a transversely recurrent train track that fully carries $\mu$ and let $\tau_1, \tau_2, \tau_3,\cdots$ be the unique sequence of train tracks, such that $\tau_{i+1}$ is a length-i unzipping of $\tau_1$ and $\tau_{i+1}$ fully carries $\mu$.  

Let  $E_i =\{\mG\in \EL|\mG$ is fully carried by $\tau_i\}$.   We verify that $E_i$ is open in $\EL$ by showing that any $\mH\in \EL$ close in the Hausdorff topology to a diagonal extension of $\mG$ is also carried by $\tau_i$.  This readily follows from the fact that $\tau_i$ is maximal, hence any diagonal extension of $\mG$ is carried by $\tau_i$.   
By Corollary \ref{splitting corollary} by dropping the first terms of the $\tau_i$ sequence, we can assume that if $\mG\in E_1$, then $\mG\subset U$.  

We next show that there exists $W_1$ open in $PT(S)$ so that $\mu\subset W_1$ and if $\mL\in \mL(S)$ with $\mL\cap W_1\neq\emptyset$, then $\mL$ is fully carried by $\tau_1$.  Since $\mu$ is minimal, if $W_1$ is a very small neighborhood and $\mL\cap W_1\neq\emptyset$, then some large compact segment $\sigma$ of a leaf of $\mL$ has the property that it is very close to $\mu$ in the Hausdorff topology, so it lies in $N(\tau_1)$, is transverse to the ties and hits each tie multiple times.   Using the maximality of $\tau_1$ isotope the rest of $\mL$ into $N(\tau_1)$ so that the resulting lamination is transverse to the ties, i.e. so that $\tau_1$ fully carries $\mL$.  

There exists $W_2$ open in $PT(S)$ so that $\mu\subset W_2$ and if $\mL\in \mL(S)$ with $\mL\cap W_2\neq\emptyset$, then \emph{every} sublamination of $\mL$ is fully carried by $\tau_1$.  If false, then there exists a sequence of laminations $\mL_1, \mL_2, \cdots$ with sublaminations $\mG_1, \mG_2, \cdots$ such that $\mL_i\cap N(\mu,1/i)\neq\emptyset$, but $\mG_i\cap W_2=\emptyset$.  After passing to subsequence we can assume the $\mG_i$'s converge in the Hausdorff topology to $ \mG$ where $\mG\cap W_2=\emptyset$.  Since $\mu$ is filling, this implies that $\mG\cap\mu\neq\emptyset$ and $\mG$ is transverse to $\mu$. Thus for $i $ sufficiently large, there is a lower bound on angle of intersection of $\mG_i $ with $\mu$.  On the other hand $\mu$ is approximated arbitrarily well by segments lying in leaves of $\mL_i$.  Thus for $i$ sufficiently large, $\mL_i$ has self intersection which is a contradiction.

Let $\epsilon_1>0$, be such that $N_{PT(S)}(\mu,\epsilon_1)\subset W_2$.  Again by Corollary \ref{splitting corollary} for $i$ sufficiently large, $\hat\mE(\tau_i)\subset N_{PT(S)}(\mu,\epsilon_1/2)$, where $\hat \mE(\tau_i)\subset PT(S)$ is the union of geodesics corresponding to elements of $\mE(\tau_i)$.  Let $j $ be one such $i$.

We complete the proof of the theorem where the role of $V$ (as 
in the first paragraph) is played by $E_j$.  Recall that $V(\tau_j)$ (resp. $ \inte (V(\tau_j$)) is the set of measured laminations carried (resp. fully carried) by $\tau_j$.   $\tau_j$ is transversely recurrent since it is an unzipping of $\tau_1$, it is recurrent since it carries $\mu$ and it is maximal since $\mu$ is maximal.  Therefore in the terminology of p. 27 \cite{PH}, $\tau_j$ is complete and hence by Lemma 3.1.2 \cite{PH} $\inte(V(\tau_j))$ is open in $\ML$.  Let $\mu_0=\mu$ and $\mu_1\in E_j$ and for $i=0,1$ pick  $\lambda_i\in \inte(V(\tau_j))$ so that 
$\phi(\lambda_i)=\mu_i$.  Each $ \lambda\in V(\tau_j)$ is determined by a transverse measure on the branches of $\tau_j$ and conversely.  Also each of $\lambda_0, \lambda_1$ induces a positive transverse measure on each branch of $\tau_j$.  Therefore we can define $e:[0,1]\to \inte (V(\tau_j))$ to be the straight line path from $\lambda_0$ to $\lambda_1$.    Let 
$f:[0,1]\to \ML$ be a PL almost filling path that closely approximates $e$ and connects its endpoints.  Closely means that for each $t$, $f(t)\in \inte (V(\tau_j))$.  Using $\epsilon=\epsilon_1/2$ apply Lemma \ref{effective 
connectivity} to obtain the path $g:[0,1]\to \EL$.  We show that 
$g([0,1]) \subset E_1\subset U$.  With notation as in Lemma \ref{effective connectivity}, if $s\in [0,1]$, then there exists $t\in [0,1]$ such that $\mL^*(t) \subset N_{PT(S)}(g^\prime(s), \epsilon)$, for some diagonal extension $g^\prime(s)$ of $g(s)$.
 Since $\mL^*(t)$ is carried by $\tau_j$, each point $x$ in  $\mL^*(t)$ is $\epsilon_1/2$ close in $PT(S)$ to a point of $\mu$ and by Lemma \ref{effective connectivity} some point $y$  of $g^\prime(s)$ is $\epsilon$ close to $x$, we conclude $d_{PT(S)}(y,z)\le \epsilon_1$ and hence $g^\prime(s)\cap W_2\neq\emptyset$.  Therefore all sublaminations of $g^\prime(s)$ are fully carried by $\tau_1$, so $g(s)$ is fully carried by $\tau_1$.  This completes the proof when $\mu$ is maximal.

The general case follows as above except that we must use the various diagonal 
extensions of $\mu$ and its associated train tracks.     Here are more details.  Let $U\subset \EL $ be open with $\mu\in U$ and 
let $\lambda\in \ML$ such that $\phi(\lambda)=\mu$.   Let $\tau_1, 
\tau_2, \cdots$ as above.  By dropping a finite number of initial terms 
we can assume that each $\tau_i$ has exactly $n$ recurrent diagonal 
extensions, $\tau_i^1,\cdots, \tau_i^n$ and for each $k$, the 
sequence $\tau^k_1,\tau_2^k,\cdots $ is a full unzipping sequence.   
Furthermore, we can find a sequence of laminations $\lambda_1^k, 
\lambda_2^k, \cdots \in \ML$ such that $\lim_{i\to\infty}\ \lambda_i^k\to 
\lambda,\ \lambda_i^k$ is fully carried by $\tau_i^k$ and $\cap \mE
(\tau_i^k)=\mE(\mu^k)$ where $\mu^k$ is a diagonal extension of $
\mu$.  By an elementary topological argument, each such $\mu^k$ is 
carried by some $\{\tau_i^q\}$ sequence where each $\tau_i^q$ is 
maximal.  (Compare with \cite {H2} which shows that a birecurrent 
train track is contained in a maximal such track.)  Let 
$E_i=\{\mG\in \EL|\mG\ \textrm{is fully carried by a}\ \tau_i^k\}$.  Each 
$E_i$ is open in $\EL$ and  by Corollary 
\ref {splitting corollary} after dropping the first $N$ terms of the various 
sequences we can assume that $E_1\subset U$.  

There exists $W_2$ open in $PT(S)$ such that $\mu\subset W_2$ and if $\mL\in \mL(S)$ with $\mL\cap W_2\neq\emptyset$, then every minimal sublamination of $\mL$ is fully carried by one of $\tau_1^1,\cdots,\tau_1^n$.  Since each diagonal of a $\mu^k$ is dense in $\mu$, there exists $W_2^1, \cdots, W_2^n$ open in $PT(S)$ such that for each $k, \ \mu^k\subset W_2^k$ and if $\mL\cap W_2^k\neq\emptyset$, then every minimal sublamination of $\mL$ is fully carried by one of $ \tau_1^1,\cdots,\tau_1^n$.   Let $\epsilon_1>0$, be such that $N_{PT(S)}(\mu^k,\epsilon_1)\subset W_2^k$ for each $k$.  Let $j$ be so that  for each $k$, $\hat\mE(\tau_j^k)\subset N_{PT(S)}(\mu^k,\epsilon_1/2)$.

Let $\mu_0=\mu$ and $\mu_1\in E_j$.  Let $\tau_j^k$ be a maximal track which carries $\mu_1$.   For $i=0,1$ pick  $\lambda_i\in V(\tau_j^k)$ so that 
$\phi(\lambda_i)=\mu_i$.  Define $e:[0,1]\to V(\tau_j^k)$ to be the straight line path from $\lambda_0$ to $\lambda_1$.    Since $V(\tau_j^k)$ is a polyhedral ball of codimension-0, and generic PL paths are almost filling, there exists
$f:[0,1]\to V(\tau_j^k)$ a PL almost filling path from $\lambda_0$ to $\lambda_1$.   Using $\epsilon=\epsilon_1/2$ apply Lemma \ref{effective 
connectivity} to obtain the path $g:[0,1]\to \EL$.  We show that 
$g([0,1]) \subset E_1\subset U$.  With notation as in Lemma \ref{effective connectivity}, if $s\in [0,1]$, then $\mL^*(t) \subset N_{PT(S)}(g^\prime(s), \epsilon)$, for some diagonal extension $g^\prime(s)$ of $g(s)$.
 Since $\mL^*(t)$ is carried by $\tau_j^k$, each point $x$ in  $\mL^*(t)$ is $\epsilon_1/2$ close in $PT(S)$ to a point of $\mu^k$ and by Lemma \ref{effective connectivity} some point $y$  of $g^\prime(s)$ is $\epsilon$ close to $x$, we conclude $d_{PT(S)}(y,z)\le \epsilon_1$ and hence $g^\prime(s)\cap W_2^k\neq\emptyset$.  Therefore each minimal sublamination of $g^\prime(s)$ is fully carried by a $\tau_1^m$, so $g(s)$ is fully carried by some $\tau_1^k$.  \qed

\section{$\EL$ is cyclic}

The main result of this section is the following

\begin{theorem}  \label{cyclic}If $S$ is as in the hypothesis of Theorem \ref{main} and $\mu\in \EL$, then there exists an embedded simple closed curve in $\EL$ passing through $\mu$.\end{theorem}

\begin{lemma} \label{arc} If $\mu\in \EL$, then there exists an embedding $f:[-1,1]\to \EL$ such that $f(0)=\mu$.   \end{lemma}

\begin{proof} Let $\lambda\in\ML$ be such that $\phi(\lambda)=\mu$.  Let $\sigma \subset \ML$ be the measures supported on $\mu$.  Let $\tau$ be a complete train track  such that $\mu\subset \inte(V_\tau)$.  By choosing a generic straight line path $f:[-1,1]\to \inte(V_\tau)\subset \ML$ with $f(0)=\lambda$ we obtain a path in $\AML$ transverse to each $\hat B_\alpha$ disjoint from each $\partial \hat B_\alpha$, such that  $\finv(\sigma)=\mu$ and $f(\partial [-1,1])\in \EL$.  Let $f^*(t)$ denote the almost minimal almost filling sublamination into $\phi(f(t))$.  

We say that two geodesic laminations \emph{cross} if they have a point of transverse intersection.  We show that if $t>0$ and $s<0$, then $f^*(t)$ crosses  $f^*(s)$.  If not, being almost minimal almost filling, they must coincide.   Since $f([0,1])$ is disjoint from the $\partial \hat B_\alpha$'s this implies that $\phi(f(t))=\phi(f(s))$.  Since $\phi(f(t))$ is carried by $V(\tau)$ the straight line segment from $f(s)$ to $f(t)$ consists of measures on the same underlying lamination.   Now either $f(t)\cup f(s)\subset \hat B_\alpha$, some $\alpha$ or $f^*(s)=f^*(t)\in \EL$.  The former contradicts the fact that $f$ is transverse to $\hat B_\alpha$, while the latter  contradicts the condition  $\finv(\sigma)=\mu$.

We next show that there exists a path $g:[0,1] \to \EL$ such that $g(0)=\mu$ and $g(1)=\phi(f(1))$ and for each $t>0$ and $s<0$, $g(t)$ crosses $f^*(s)$.  Let $1=t_0> t_1> \cdots$ be a sequence such that $\Lim t_i=0$ and for all $i$, $\phi(f(t_i))\in \EL$.  It follows from Lemma \ref{effective connectivity} that for  each $i\ge 1$, there exists a path $g_i:[t_{i},t_{i-1}]\to \EL$ with $g_i(t_j)=\phi(f(t_j))$, $j\in \{i,i-1\}$ such that  for $s<0$ and $t\in [t_i, t_{i-1}]$, $g_i(t )$ crosses $f^*(s)$.  Lemma \ref{effective connectivity} together with Proposition \ref{super convergence} and the minimality of $\mu$ implies that the $g_i$'s can be further chosen so that given a neighborhood basis $\{U_k\}$ of $\mu$ in $\EL$ and $n>0$, there exists $N(n)$ so that $i>N(n)$ implies that $g([t_i, t_{i-1}])\in U_N$.  By concatenating these $g_i$'s together and decreeing $g(0)=\mu$ we obtain the desired function $g$.

By Lemma \ref{wilder} stated below, there exists an embedded path $h:[0,1] \to g([0,1])\subset\EL$ from $\mu$ to $g(1)$.  Since each $h(t)$ crosses $\phi(f^*(s))$ for $s<0$, we can repeat this argument to produce  an embedded path $h:[-1,0] \to \EL$ from $\phi(f(-1))$ to $\mu$ such that if $s<0$ and $t>0$, then $h(s)$ crosses  $h(t)$.  It follows that $h:[-1,1]\to \EL$ is an embedded path through $\mu$.  \end{proof}

\begin{lemma}  \label{wilder} If $X$ is Hausdorff and $g:[0,1]\to X$ is continuous and $p,q\in g([0,1])$,  then there exists an embedded path in $g[0,1]$ from $p$ to $q$.\end{lemma}  

\begin{proof}  This is  Corollary III.3.11 \cite{Wi} which asserts that any two points of a Peano continuum $C$ are connected by an embedded path in C, where a Peano continuum is the image of a continuous map of $[0,1]$ into a Hausdorff space.  See III.1 and III.1.29 \cite{Wi} for this characterization of Peano continuum.\end{proof} 

\noindent\emph{Proof of Theorem \ref{cyclic}}.  Apply Lemma \ref{arc} to find an embedded path $h:[-1,1]\to \EL$ with $h(0)=\mu$.  Let $\lambda_t\in \ML$ be such that $\phi(\lambda_t)=h(t)$ for $t\in \{-1,1\}$.  Let $\tau$ be a complete train track such that $\mu\in \inte (V(\tau))$ and for $t\in \{-1,1\}$, $\lambda_t\notin V(\tau)$.  Since dim$(\ML)>2$ and $V(\tau)$ is a cone on a closed simplex, it follows that there exists a path $f:[-1,1]$ in $\ML$ from $\lambda_{-1}$ to $\lambda_1$ disjoint from $V(\tau)$.  After a small homotopy we can assume that  that $f$ is a PL almost filling path.  For each $t\in [-1,1]$, $f^*(t)$ crosses $\mu$ so by Lemma \ref{super convergence} there exists $\epsilon>0$ so that for $|s|<\epsilon, t\in [-1,1]$, $f^*(t)$ crosses $h(s)$.  The proof of Lemma \ref{arc} produces an embedded path $k:[-1,1]\to \EL$ from $h(-1)$ to $h(1)$ with this same property.  Concatenating this path with $h|[-1,-\epsilon] $ and $h|[\epsilon,1]$ we obtain a path $\kappa:[a,b]\to \EL$ from $h(-\epsilon)$ to $h(\epsilon)$ which intersects $h|[-\epsilon,\epsilon]$ only at its endpoints.  Apply Lemma \ref{wilder} to obtain an embedded path with the same properties.  Fusing this path with $h|[-\epsilon,\epsilon]$ we obtain the desired embedding of $S^1$ into $\EL$ passing through $\mu$.\qed

\begin{corollary} $\EL$ has no cut points.\end{corollary}
\section{questions}

\begin{question}  \label{question 1}If $\lambda$ and $\lambda^\prime \in \PML$ satisfy $\phi(\lambda)$, $\phi(\lambda^\prime)\in \EL$, is there 
a filling path in $\PML$ connecting them?\end{question}

\begin{question}  Is $\EL$ n-connected if S has sufficiently high 
complexity?\end{question}

\newpage

\section{Applications} 

An unsuccessful attempt to find a positive solution to Question \ref{question 1} led to the interesting study of how various  $B_\beta$'s can intersect a given $B_\alpha$ which in turn led to a new construction of non uniquely ergodic measured laminations in any finite type surface S not the 1-holed torus, or 3 or 4-holed sphere.

\begin{theorem}\label{uniquely ergodic}   If $\alpha_1, \cdots, \alpha_n$  are pairwise disjoint simple closed geodesics in the complete, hyperbolic, finite type, orientable surface $S$, and for each $i$, $U_i$ is an open set in $\PML$ about $\alpha_i$ then there exists a (n-1)-simplex $\Sigma$ in $\PML$ representing a filling non uniquely ergodic measured lamination with the endpoints of $\Sigma$ respectively in $U_1, \cdots, U_n$.   \end{theorem}

\begin{proof}  In the statement of the theorem as well as in the proof, simple closed curves are naturally identified with points in $\PML$.  

We provide the argument for the case n=2.  The proof in the general case is similar.  Let $\alpha$ and $\beta$ (resp. $U_\alpha$ and $U_\beta$) denote $\alpha_1$ and $\alpha_2$ (resp. $U_1$ and $U_2$).  Let $a_1$ be a simple closed geodesic disjoint from $\beta$ and close 
to $\alpha$.   For example, let $a^\prime_1\neq\alpha$ be any geodesic 
disjoint from $\beta$ which intersects $\alpha$ and let $a_1$ be obtained 
from $a^\prime_1$ by doing a high power Dehn twist to $\alpha$.    Let 
$b_1$ a simple closed geodesic disjoint from $a_1$ and very close to 
$\beta$.  Let  $a_2$ a simple closed geodesic disjoint from $b_1$ and 
extremely close to $a_1$ and let $b_2$ a simple closed geodesic 
disjoint from $a_2$ and super close to $b_1$.  Continuing in 
this manner we obtain the sequence $a_1, a_2, a_3, \cdots$ which 
converges to $\lambda_\alpha$ and the sequence $b_1, b_2, \cdots$ 
converges to $\lambda_\beta$ where $\lambda_\alpha, \lambda_\beta
$ respectively lie in $U_\alpha$ and $U_\beta$.    

Since intersection number is continuous on $\ML$, and for each $i$,  $a_i$ and $b_i$ are pairwise disjoint it follows that $i(\lambda_\alpha,\lambda_\beta)=0$.  We now show that with a little care in the choice of the $a_i$'s and $b_i$'s  that $\phi(\lambda_\alpha)=\phi(\lambda_\beta)\in \EL$.   Let $V_\alpha\subset U_\alpha$ and $V_\beta\subset U_\beta$ be closed neighborhoods of $\alpha$ and $\beta$.  We argue as in the proof of Lemma \ref{path criterion}. 

Recall that $C_1, C_2, \cdots$ is   an enumeration of the simple closed geodesics  in  $S$.  If possible choose $a_1$ to intersect $C_1$ transversely.  We can always do that  and have $a_1$ as close to $\alpha$ as we like, unless $C_1=\beta$ or $\beta$ separates $C_1$ from $\alpha$.  If we must have $a_1\cap C_1=\emptyset$, then we can choose $b_1\in \inte(V_\beta)$ so that $b_1\cap C_1\neq\emptyset$ and hence in any case, $a_2$ can be chosen so that $a_2\cap C_1\neq \emptyset$.  Let $V_2\subset V_\alpha$ a closed neighborhood of $a_2$ in $\PML$ so that if $x\in V_2$, then $x$ and $C_1$ have non trivial intersection number.  As above, choose $b_2,\ a_3,\ b_3,\ a_4$ so that $a_4\cap C_2\neq\emptyset,\ a_3\cup a_4\subset \inte(V_2)$ and $b_2\cup b_3\subset \inte(V_\beta)$.  Choose $V_4\subset V_2$ a closed neighborhood of $a_4$ so that $x\in V_4$ implies that $x$ and $C_2$ have nontrivial intersection number.  Inductively construct $b_4, b_5,\cdots,\ a_5, a_6,\cdots$.  If $\lambda_\alpha=\Lim a_i$, then for all $j,\ i(\lambda_\alpha, C_j)\neq 0$ and hence $\phi(\lambda_\alpha)\in \EL$.  By construction $\lambda_\alpha\in U_\alpha$ and $\lambda_\beta=\Lim b_i\in U_\beta$.  Since $\lambda_\alpha\in \EL$ and has trivial intersection number with $\lambda_\beta$, it follows that $\phi(\lambda_\alpha)=\phi(\lambda_\beta)$.  The projection to $\PML$ of a straight line segment in $\ML$ connecting representatives of $\lambda_\alpha$ and $\lambda_\beta$  gives the desired 1-simplex $\Sigma$.\end{proof}

\vskip 8pt

In \cite{RS} Kasra Rafi and Saul Schleimer derived a number of important rigidity results about a surface $S$ and its curve complex $\CS$ under the assumption that $\CS$ is connected.  In particular an immediate consequence of Theorem 7.1 and Corollary 1.2 of  \cite{RS} and Theorem \ref{main} is the

\begin{theorem} \label{simplicial} Let $S$ be a finite type orientable surface of negative Euler chacteristic which is not the 3-holed sphere, 4-holed sphere or 1-holed torus.  Every quasi-isometry of $\CS$ is bounded distance from a simplicial automorphism of $\CS$.  Consequently, QI$(\CS)$ the group of quasi-isometries of $\CS$ is isomorphic to Aut$(\CS)$ the group of simplicial automorphisms.  \end{theorem}

An immediate consequence of Theorem \ref{main} and the quasi-isometric rigidity Theorem 1.4 of Rafi -- Schleimer \cite{RS} is the

\begin{theorem}  \label{quasi} Let $S$ be a finite type orientable surface of negative Euler chacteristic which is not the 3-holed sphere, 4-holed sphere or 1-holed torus.  Suppose that neither $S$ nor $\Sigma$ is the surface of genus-2 or the twice punctured surface of genus-1, then $S$ and $\Sigma$ are homeomorphic if and only if $\CS$ is quasi-isometric to $\mC(\Sigma)$.\end{theorem}

\vskip 10 pt

In \cite{LS} Chris Leininger and Saul Schleimer give a detailed proof of the following Theorem \ref{DD} which they say "seems to be well known", that is a consequence of many deep results in Kleinian group theory.  Using their notation, the space of \emph{doubly degenerate Kleinian surface} groups is denoted by DD$(S, \partial S)$ and equals $\{M\in \text{AH}(S,\partial S)|\mE(M)\in \EL\times \EL\}$.  Here $\mE(M)$ is the end invariant of the doubly degenerate group $M$ and is an element of $\EL\times \EL-\Delta$, where $\Delta$ is the diagonal.  

\begin{theorem}  \label{DD} (Theorem 6.5 \cite{LS}) $\E:$DD$(S,\partial S)\to \EL\times \EL\setminus\Delta$ is a homeomorphism.\end{theorem}

\begin{theorem}  If $S$ is a finite type hyperbolic surface, not the once 
punctured torus or 3 or 4 times punctured sphere, then $\EL\times \EL\setminus \Delta$ is connected, locally path connected and cyclic.
\end{theorem}

\begin{proof}  Local path connectivity is immediate.  Let $\sigma_0, \sigma_1,\mu$ be distinct elements of  $\EL$.  The method of proof of Theorem \ref{cyclic} shows how to construct a path from $\sigma_0$ to $\sigma_1$ disjoint from $\mu$ and an embedded circle through $\sigma_0$ disjoint from $\mu$.  Using this it is routine to show path connectivity and cyclicity.\end{proof}

\begin{corollary}  \label{DD connectivity} DD$(S,\partial S)$ is connected, path connected and cyclic if $S$ is a compact hyperbolic surface that is not the sphere with 3 or 4 open discs removed or the torus with an open disc removed.\qed\end{corollary}

\enddocument